\documentclass{amsart}
\usepackage[left=2cm,right=2cm,top=2cm,bottom=1.5cm]{geometry}
\usepackage{longtable} 
\usepackage{amsmath,amssymb,amsthm}
\usepackage{enumerate} 
\usepackage{array}
\newcolumntype{C}[1]{>{\centering\arraybackslash}m{#1}}
\newcolumntype{L}[1]{>{\raggedright\arraybackslash}m{#1}}

\newtheorem{theorem}{Theorem}[section]
\newtheorem{lemma}[theorem]{Lemma}
\newtheorem{proposition}[theorem]{Proposition}
\newtheorem{corollary}[theorem]{Corollary} 
\theoremstyle{definition}

\theoremstyle{remark}

\numberwithin{equation}{section}

\newcommand{\name}{A}
\newcommand{\newname}{B}
\renewcommand{\a}{\mathrm a}
\renewcommand{\b}{\mathrm b} 
\renewcommand{\c}{\mathrm c} 

\newcommand{\ve}{\varepsilon} 

\newcommand{\Ann}[1]{{#1}^0} 
\newcommand{\e}{\mathrm{e}}

\newcommand{\ccdot}{\!\cdot\!}  
\newcommand{\hook}{\lrcorner \,}

\newcommand{\yes}{\checkmark}

\newcommand{\fett}[1]{\textbf{\underline{#1}}}
\newcommand{\nil}[1]{\mathrm{Nil}(#1)}

\newcommand{\bR}{\mathbb{R}}

\renewcommand{\Re}{\mathrm{Re}}

\newcommand{\spa}[1]{\mathrm{span}(#1)}

\newcommand{\n}{\mathfrak{n}}

\newcommand{\g}{\mathfrak{g}} 
\newcommand{\h}{\mathfrak{h}}

\newcommand{\G}{\mathrm{G}}
\newcommand{\GL}{\mathrm{GL}}

\newcommand{\SU}{\mathrm{SU}}

\newcommand{\op}{\oplus} 
\newcommand{\ot}{\otimes}
\newcommand{\ad}{\mathrm{ad}}
\newcommand{\tr}{\mathrm{tr}}
\renewcommand{\o}{\omega}  

\renewcommand{\L}{\Lambda}
\newcommand{\s}{\sigma} 
\renewcommand{\^}{\wedge}

\begin{document}
\title{Half-flat structures on indecomposable Lie groups}
\author{Marco Freibert}
\address{Marco Freibert, Fachbereich Mathematik, Universit\"at
  Hamburg, Bundesstr.~55, 20146 Hamburg, Germany}
\email{freibert@math.uni-hamburg.de}

\author{Fabian Schulte-Hengesbach}
\address{Fabian Schulte-Hengesbach, Fachbereich Mathematik, 
Universit\"at Hamburg, Bundesstr.~55, 20146 Hamburg, Germany}
\email{schulte-hengesbach@math.uni-hamburg.de}

\subjclass[2000]{53C25 (primary), 53C15, 53C30 (secondary)}

\begin{abstract}
  This article can be viewed as a continuation of the articles
  \cite{SH} and \cite{FS} where the decomposable Lie algebras
  admitting half-flat $\SU(3)$-structures are classified. The new main
  result is the classification of the indecomposable six-dimensional
  Lie algebras with five-dimensional nilradical which admit a
  half-flat $\SU(3)$-structure. As an important step of the proof, a
  considerable refinement of the classification of six-dimensional Lie algebras with
  five-dimensional non-Abelian nilradical is established. Additionally, it is proved
  that all non-solvable six-dimensional Lie algebras admit half-flat
  $\SU(3)$-structures.
\end{abstract}

\maketitle

\section{Introduction}
\label{intro}
$\SU(3)$-structures on a real six-manifold $M$ are reductions of the frame bundle of $M$
to $\SU(3)$ and can equivalently be described
as quadruples $(g,J,\o,\Psi)$ consisting of an almost Hermitian structure $(g,J,\o)$
and a unit $(3,0)$-form $\Psi$. Such a structure is called half-flat
if it satisfies
\begin{equation*}
\label{hf}
  d \,\Re \Psi=0\;, \quad d (\o \wedge \o)=0.
\end{equation*}
A left-invariant half-flat $\SU(3)$-structure on a  six-dimensional Lie group can be characterised by a pair
$(\o,\rho)\in \L^2\g^*\times \L^3 \g^*$
of a non-degenerate two-form $\o$ and a three-form $\rho$ of specific type on the associated Lie algebra $\g$
satisfying certain compatibility relations and $d\rho=0=d (\o \wedge \o)$.
We say that such a pair $(\o,\rho)$ is a half-flat $\SU(3)$-structure on the Lie algebra $\g$.

The problem of determining the six-dimensional Lie algebras which admit a half-flat $\SU(3)$-structure
has been solved for the nilpotent case by Conti \cite{C}. In \cite{SH}, one of the authors
has classified direct sums of two three-dimensional Lie algebras which admit
such structures. The remaining decomposable six-dimensional Lie algebras
which admit half-flat $\SU(3)$-structures have been classified by the authors \cite{FS}. 

In this article we continue the work of \cite{FS} and tackle the
problem for indecomposable six-dimensional Lie algebras. Non-solvable
indecomposable six-dimensional Lie algebras have been classified by
Turkowski \cite{Tu}. The resulting list is given in Table
\ref{table_6dnonsolv} also including the only indecomposable
six-dimensional simple Lie algebra.  We obtain the following result by
giving explicit examples in all cases in Table
\ref{examples_6dnonsolv} and by considering the known results in the
decomposable case \cite{SH}, \cite{FS}.
\begin{theorem}\label{Th2}
Let $\g$ be a non-solvable six-dimensional Lie algebra. Then $\g$ admits a half-flat $\SU(3)$-structure.
\end{theorem}

In the solvable case, we restrict our attention to indecomposable
six-dimensional Lie algebras with five-dimensional nilradical.  For
the algebras with Abelian nilradical, we can exploit the relation
between half-flat $\SU(3)$-structures and cocalibrated
$\G_2$-structures and solve the existence problem by using a result of
the first author \cite{F}. For the algebras with non-Abelian
nilradical, we have to resort to the classification of the
indecomposable six-dimensional Lie algebras with five-dimensional
nilradical. The slightly cumbersome classification was first
established by Mubarakzyanov \cite{Mu} in 1963 and was recently
corrected by Shabanskaya \cite{Sha}.

In order to prove the main classification result, we need to subdivide
Mubarakzyanov's classes according to the dimensions $\mathrm{h}^k(\g)$
of the Lie algebra cohomology groups and the dimension of the centre
similar as in \cite{FS}.  Although we exclude the Lie algebras with
Abelian nilradical, the number of naturally appearing subclasses turns
out to be larger than we expected, cf.\ Table
\ref{table_6d_5dNR}. Nevertheless, the new data contributes to the
understanding of low-dimensional Lie algebras and may be useful in the
study of further problems concerning six-dimensional Lie algebras. As
a first application, apart from our main result, we classify the
(2,3)-trivial six-dimensional Lie algebras in section
\ref{sec:23trivial}.

Since the Lie algebras with five-dimensional nilradical admitting a half-flat $\SU(3)$-structure have hardly
anything in common, a simple characterisation seems not possible and we have to state our main result in the following form. 
\begin{theorem}\label{Th1}
A solvable indecomposable six-dimensional Lie algebra with five-dimensional nilradical admits a half-flat $\SU(3)$-structure if and
only if it is contained in Table \ref{examples6d_5dNR}.
\end{theorem}
\noindent Note that Table \ref{examples6d_5dNR} also contains an explicit example for each case.

By a result of Mubarakzyanov \cite{Mu}, a solvable six-dimensional Lie
algebra which is neither nilpotent nor decomposable has four- or
five-dimensional nilradical. Hence, the classification remains open
only for the case of solvable indecomposable six-dimensional Lie
algebras with four-dimensional nilradical.

This work was supported by the German Science Foundation (DFG) within
the Collaborative Research Centre 676 "Particles, Strings and the
Early Universe". The authors also thank the university of Hamburg for
financial support, Vicente Cort\'es for suggesting the project and
for pointing out the application concerning (2,3)-trivial Lie
algebras and Anna Fino for pointing out the new classification \cite{Sha} of six-dimensional
Lie algebras with five-dimensional nilradical.

\section{Obstruction theory for half-flat $\SU(3)$-structures}
\label{sec:obst}
\subsection{Known obstructions}
A half-flat $\SU(3)$-structure on a Lie algebra $\g$ can be described
within the framework of stable forms on real vector spaces developed by Hitchin \cite{Hi1} and thoroughly discussed 
e.g. in \cite{CLSS}. 

A \emph{stable} $k$-form $\rho\in \L^k V^*$ on a vector space $V$ is a $k$-form which lies in 
an open $\GL(V)$ orbit for the natural action of $\GL(V)$ on $V$. A two-form in even dimensions is stable if and only 
if it is non-degenerate. To characterise stability of three-forms $\rho\in \L^3 V^*$ on an oriented six-dimensional vector 
space $V$, Hitchin \cite{Hi2} introduced
\begin{equation*}
K_{\rho}(v):=\kappa\left((v\hook\rho)\wedge \rho \right)\in V\ot \Lambda^6 V^*,\quad \lambda(\rho):=\frac{1}{6} \tr K_{\rho}^2 \in \left(\Lambda^6 V^*\right)^{\ot 2},\quad J_{\rho}:= \frac{1}{\sqrt{|\lambda(\rho)|}} K_{\rho}\in V\ot V^*,
\end{equation*}
where $\kappa$ denotes the natural isomorphism $\Lambda^5 V^* \cong V \otimes \Lambda^6 V^*$ and $\sqrt{|\lambda(\rho)|}\in \Lambda^6 V^*$ denotes the positively oriented root of $|\lambda(\rho)|\in \left(\Lambda^6 V^*\right)^{\ot 2}$. A three-form $\rho\in \L^3 \g^*$ in dimension six is stable if and only if $\lambda(\rho)\neq 0$ and the induced endomorphism $J_{\rho}$ of $V$ is a complex structure on $V$ if and only if $\lambda(\rho)<0$ .

A \emph{half-flat $\SU(3)$-structure} on a six-dimensional Lie algebra $\g$ is a pair of stable forms $(\omega,\rho)\in \L^2 \g^*\times \L^3 \g^*$ such that $\lambda(\rho)<0$, $\omega\wedge \rho=0$ and $d\o^2=0=d\rho$ and such that $g(\cdot,\cdot):=\omega(J_{\rho}\cdot,\cdot)$ is a Euclidean metric. A half-flat $\SU(3)$-structure $(\omega,\rho)\in \L^2 \g^*\times \L^3 \g^*$ is called \emph{normalised} if $J_\rho^* \rho \wedge \rho=\frac{2}{3}\, \omega^3$.

In \cite[Corollary 3.2]{FS} we already developed an obstruction to the existence of a half-flat $\SU(3)$-structure
which is a refinement of the obstructions developed by Conti \cite{C}. We state it here again denoting by
$Z^p(\g)$ the space of all closed $p$-forms on $\g$.
\begin{proposition}
  \label{obst_algebra1}
  Let $\g$ be a six-dimensional Lie algebra with a volume form $\nu
  \in \L^6\g^*$. If there is a non-zero one-form $\alpha \in\g^*$
  satisfying
  \begin{equation}
  \label{tildeJ:1}
    \alpha \^ \tilde J^*_\rho \alpha \^ \s=0
  \end{equation} 
  for all $\rho \in Z^3(\g)$ and all $\s \in Z^4(\g)$,
  where $\tilde J_\rho^*\alpha $ is defined for $X \in \g$ by 
   \begin{eqnarray}
  \label{tildeJ:2}
  \tilde J_{\rho}^* \alpha (X) \,\nu &=& 
  \alpha \wedge (X \hook \rho) \wedge \rho,
  \end{eqnarray}
  then $\g$ does not admit a half-flat
  $\SU(3)$-structure.
\end{proposition}

\subsection{Obstructions from the relation to cocalibrated $\G_2$-structures}\label{subsec:G2}
The obstruction obtained above is not the only tool we need for the proof of Theorem \ref{Th1}.
Additionally, we exploit the relation between half-flat $\SU(3)$-structures
on six-dimensional Lie algebras $\g$ and cocalibrated $\G_2$-structures on seven-dimensional Lie algebras
$\g\op \bR$. We roughly sketch this relation which is discussed in detail e.g. in \cite{CLSS} and \cite{S}.

Let $(\o,\rho)\in \L^2\g^*\times \L^3 \g^*$ be a 
half-flat $\SU(3)$-structure on a six-dimensional Lie algebra $\g$. Denote by
$J_{\rho}$ the complex structure on $\g$ induced by $\rho$. Choose some
non-zero element $\alpha\in \Ann{\g}\backslash\{0\}$ in the annihilator $\Ann{\g}$ of $\g$ in $\h:=\g\op \bR$.
Then
\begin{equation*}
\varphi:=\o\wedge \alpha +J_{\rho}^*\rho
\end{equation*}
is a \emph{$\G_2$-structure} on $\h:=\g\op \bR$, i.e. the stabiliser of
$\varphi$ under the natural action of $\GL(\h)$ on $\L^3 \h^*$
is isomorphic to $\G_2$. Since $\G_2\subseteq \mathrm{SO}(7)$, the $\G_2$-structure $\varphi$ induces a Euclidean
metric $g_{\varphi}$ and an orientation on $\h$ such that the subspaces $\g$ and $\bR$ are orthogonal with
respect to $g_{\varphi}$. Hence, it induces a Hodge star operator $\star_{\varphi}:\L^3 \h^*\rightarrow \L^4 \h^*$.
A straightforward computation shows that the Hodge dual $\star_{\varphi}\varphi$
of $\varphi$ is given by
\begin{equation}\label{eq:Hodgedual}
\star_{\varphi}\varphi=\rho\wedge \alpha+\frac{1}{2}\o^2.
\end{equation}
Since the forms $\rho$, $\frac{1}{2}\o^2$ and $\alpha\in \g^0$ are all closed, the four-form $\star_{\varphi} \varphi$
is closed, as well. A $\G_2$-structure $\varphi$ with closed Hodge dual $\star_{\varphi}\varphi$ is called \emph{cocalibrated}.

Thus, each half-flat $\SU(3)$-structure on a six-dimensional Lie algebra induces a cocalibrated $\G_2$-structure
on the seven-dimensional Lie algebra $\h=\g\op \bR$.
\begin{proposition}\label{Abeliannilradical}
Let $\g$ be an indecomposable six-dimensional Lie algebra and
assume that the nilradical $\nil \g$ of $\mathfrak{g}$ is isomorphic
to $\bR^5$. Then $\g$ does not admit a half-flat $\SU(3)$-structure.
\end{proposition}
\begin{proof}
Let $\g$ be an indecomposable six-dimensional Lie algebra with five-dimensional Abelian nilradical. Assume that $\g$ admits a
half-flat $\SU(3)$-structure. Then the seven-dimensional Lie algebra $\h:=\g\op \bR$ has to admit a cocalibrated $\G_2$-structure. 
In \cite{F}, the seven-dimensional Lie algebras $\h$ with Abelian codimension one ideals which admit cocalibrated $\G_2$-structures
have been classified. Note that $\mathfrak{u}:=\nil \g\op \bR$ is an Abelian ideal of codimension one in $\h$ which allows us to apply the results of \cite{F} as follows. The indecomposability of $\g$ implies that the complex Jordan normal form of $\ad(e_7)|_{\mathfrak{u}}$ for some $e_7\in \h \backslash \mathfrak{u}$ has to contain exactly one Jordan block of size one with $0$ on the diagonal. However, this statement contradicts \cite[Theorem 1.1]{F} stating that in the complex Jordan normal form of $\ad(e_7)|_{\mathfrak{u}}$ the number of Jordan blocks of size one with zero on the diagonal has to be even.
\end{proof}
In the proof of the main result, we use another obstruction obtained by the relation between half-flat $\SU(3)$-structures and cocalibrated $\G_2$-structures. The following result on $\G_2$-structures can be found in \cite[Lemma 3.9, Remark 3.10]{F}.
\begin{lemma}\label{le:obst}
Let $V$ be a seven-dimensional vector space and $\varphi\in \L^3 V^*$ be a $\G_2$-structure on $V$. Let $v\in V$, $v\neq 0$ be arbitrary and $U$ be an arbitrary complement of $\spa{v}$ in $V$. Then the three-form
$\tilde{\rho}:=(v\hook\star_{\varphi} \varphi)|_{U}\in \L^3 U^*$ is a stable three-form on $U$ with $\lambda(\tilde{\rho})<0$.
\end{lemma}
\noindent As a consequence, we can prove the following obstruction condition.
\begin{proposition}
\label{obst_algebra2}
Let $\g$ be a six-dimensional Lie algebra and set $\h:=\g\op \bR$. Choose a non-zero one-form
$\alpha\in \Ann{\g}$ in the annihilator $\Ann{\g}$ of $\g$ in $\h$. For each pair $(\rho,\s)\in Z^3(\g)\times Z^4(\g)$
of a closed four-form and a closed three-form on $\g$ we define a four-form $\Omega(\rho,\s)\in \L^4 \h^*$ on $\h$
as follows:
\begin{equation*}
\Omega(\rho,\s):=\rho\wedge \alpha+\s.
\end{equation*}
If there exists a non-zero element $X\in \h$ and a complement $W$ of $\spa{X}$ in $\h$ such that for all pairs $(\rho,\s)\in Z^3(\g)\times Z^4(\g)$ the three-form $\tilde{\rho}(\rho,\s):= \left(X\hook\Omega(\rho,\s)\right)|_{W}\in \L^3 W^*$ on $W$ fulfils $\lambda(\tilde{\rho})\geq 0$, then $\g$ does not admit any half-flat $\SU(3)$-structure.
\end{proposition}
\begin{proof}
Let $\g$, $\h$, $\alpha\in \Ann{\g}$ as in the statement. Assume that $X\in \h$ and $W\subseteq \h$ as in the statement exist and that, nevertheless, $\g$ admits a half-flat $\SU(3)$-structure $(\omega,\rho)\in \L^2 \g^*\times \L^3 \g^*$. Set $\sigma:=\frac{1}{2}\o^2$. Then $(\rho,\s)\in Z^3(\g)\times Z^4(\g)$. As aforementioned, the half-flat $\SU(3)$-structure $(\omega,\rho)$ induces a cocalibrated $\G_2$-structure $\varphi$ on $\h$ and by equation (\ref{eq:Hodgedual}), the Hodge dual is given by
\begin{equation*}
\star_{\varphi} \varphi=\rho\wedge \alpha+\s=\Omega(\rho,\s).
\end{equation*}
Applying Lemma \ref{le:obst} for $0\neq X\in \h$ and the complement $W$ of $\spa{X}$ in $\h$ yields that the three-form
$\tilde{\rho}(\rho,\s)=(X\hook \star_{\varphi} \varphi)|_W=(X\hook\Omega(\rho,\s))|_W\in \L^3 W^*$ on $W$ has to fulfil $\lambda(\tilde{\rho})<0$, a contradiction.
\end{proof}
\section{Proof of Theorem \ref{Th1}}
First of all, Proposition \ref{Abeliannilradical} yields that all Lie
algebras with Abelian nilradical do not admit a half-flat
$\SU(3)$-structure. As Table \ref{table_6d_5dNR} contains all
indecomposable Lie algebras with non-Abelian five-dimensional
nilradical according to the classification of Mubarakzyanov \cite{Mu}
and Shabanskaya \cite{Sha}, it suffices to prove existence or
non-existence in each case contained in the list. The existence
problem is completely solved by the explicit examples given in Table
\ref{examples6d_5dNR}. In the following, we prove the non-existence
for the remaining Lie algebras.
  
For all Lie algebras, except $\name_{6,39}^{-1,-1}$,
$\name_{6,41}^{-1}$ $\name_{6,76}^{-1}$, $\name_{6,78}$,
$\newname_{6,3}$ and $\newname_{6,4}^0$, we apply Proposition
\ref{obst_algebra1}.  In each case, we work in the basis
$(e^1,\ldots,e^6)$ of $\g^*$ given in Table \ref{table_6d_5dNR}.
Analogously to the proof of \cite[Proposition 3.3]{FS}, we show that
$\alpha=e^6$ is a one-form fulfilling equation (\ref{tildeJ:1}) for all
$\rho \in Z^3(\g)$ and all $\s \in Z^4(\g)$.  More precisely, we start
with a pair $(\rho,\sigma)\in \L^3 \g^*\times \L^4\g^*$ of a
three-form $\rho$ and a four-form $\sigma$ expressed with respect to
the induced basis on forms using 35 coefficients in total. The general
solution of the equations $d\rho=0$ and $d\omega=0$ can be obtained by
eliminating a certain amount of coefficients due to the separation of
the classes in Table \ref{table_6d_5dNR} by different Lie algebra
cohomology. The computation of $\tilde{J}_{\rho}$ with respect to the
given basis by equation (\ref{tildeJ:2}) allows us to verify equation
(\ref{tildeJ:1}) for $\alpha=e^6$ and all $(\rho,\s)\in Z^3(\g)\times
Z^4(\g)$. All calculations can be executed conveniently in a computer
algebra system.
  
Unfortunately, Proposition \ref{obst_algebra1} cannot be applied to
the Lie algebras $\name_{6,39}^{-1,-1}$, $\name_{6,41}^{-1}$
$\name_{6,76}^{-1}$, $\name_{6,78}$, $\newname_{6,3}$ and
$\newname_{6,4}^0$. The following proof uses Proposition
\ref{obst_algebra2}.  Again, we compute the general closed three-form
$\rho\in Z^3(\g)$ and the general closed four-form $\s\in Z^4(\g)$
with respect to the basis $(e^1,\ldots,e^6)$ given in Table
\ref{table_6d_5dNR}.  We choose $e^7\in (\g\op\bR)^*$ with $d e^7=0$
such that $(e^1,\ldots,e^7)$ is a basis of $(\g\op\bR)^*\cong \g^*\op
\bR^*$ and compute $\Omega(\rho,\s):=\rho\wedge e^7+\s\in \L^4 (\g\op
\bR)^*$.  Defining $X:=e_3$ and $W:=\spa{e_1,e_2,e_4,e_5,e_6,e_7}$, we
compute for each of the four Lie algebras the three-form
$\tilde{\rho}(\rho,\s)\in \L^3 W$,
$\tilde{\rho}(\rho,\s):=\left(X\hook\Omega(\rho,\s)\right)|_{W}$.
When we compute $\lambda(\tilde{\rho}(\rho,\s))$ with respect to $W$,
it turns out that it is in each case the square of a polynomial in the
coefficients of the general closed three-form $\rho\in Z^3(\g)$ and of
the general closed four-form $\s\in Z^4(\g)$.  Thus
$\lambda(\tilde{\rho})\geq 0$ and none of the four Lie algebras admits
a half-flat $\SU(3)$-structure according to Proposition
\ref{obst_algebra2}.
  
  \section{$(2,3)$-trivial Lie algebras of dimension six}\label{sec:23trivial}
An interesting application of Table \ref{table_6d_5dNR} is the classification of six-dimensional Lie algebras which are \emph{$(2,3)$-trivial}, i.e. whose second and third Lie algebra cohomology vanishes. These Lie algebras play an analogous role for the study of multi-moment maps as semi-simple Lie algebras do for the study of moment maps in symplectic geometry, see \cite{MS1} and \cite{MS2}.

A classification of $(2,3)$-trivial Lie algebras up to dimension five has been established by Madsen and Swann in \cite{MS1}. The most important tool for the classification is the following theorem proved in \cite{MS2}.
\begin{theorem}[Madsen, Swann]\label{th:madsenswann}
A Lie algebra $\g$ is $(2,3)$-trivial if and only if $\g$ is solvable, the derived Lie algebra $\mathfrak{n}=[\g,\g]$ is nilpotent of codimension one in $\g$ and $H^i(\mathfrak{n})^{\g}=\{0\}$ for $i=1,2,3$.
\end{theorem}

In particular, $(2,3)$-trivial Lie algebras are indecomposable. Thus, Theorem \ref{th:madsenswann}
implies the following result on $(2,3)$-trivial six-dimensional Lie algebras.
\begin{corollary}\label{cor:23trivial}
A six-dimensional Lie algebra $\g$ is $(2,3)$-trivial if and only if it is one of the Lie algebras in Table \ref{table_6d_5dNR} with $h^2(\g)=h^3(\g)=0$ or if the nilradical $\mathfrak n$ of $\g$ is $\bR^5$ and the induced endomorphism $\ad(v)\,|_{\L^i \mathfrak n}$ for an arbitrary $v\in \g\backslash \mathfrak n$ has trivial kernel for $i=1,2,3$.
\end{corollary}

\section*{Appendix}
Table \ref{table_6dnonsolv} contains all non-solvable indecomposable
six-dimensional Lie algebras and Table \ref{table_6d_5dNR} contains
all indecomposable six-dimensional Lie algebras with five-dimensional,
non-Abelian nilradical. Table \ref{table_6d_5dNR} is further
subdivided by the different non-Abelian nilradicals which appear.

The notation and the Lie brackets in Table \ref{table_6dnonsolv} are
taken literally from \cite{Tu}. Table \ref{table_6d_5dNR} is based on
the original list by Mubarakzyanov \cite{Mu} and, apart from the
obvious subdivision according to the number of free parameters and the
Lie algebra cohomology, the list is modified as follows. On the one
hand, some of Mubarakzyanov's classes $g_{6,n}$ are redundant since
there is an isomorphism to one of the other classes for certain
parameter values. On the other hand, Shabanskaya \cite{Sha} found 6
new classes which are fitted in Table \ref{table_6dnonsolv} according
to their nilradical and denoted by $\newname_{6,i}$, $i=1,\dots,6$.
Moreover, a large number of isomorphisms for certain parameter values
has been discovered by Shabanskaya \cite{Sha} and the authors
resulting in a range restriction or vanishing of certain
parameters. It turns out to be hard to assure that no further
isomorphisms are possible due to the complexity and large amount of
data. Lastly, a few parameter values are excluded because the corresponding
Lie algebra is decomposable or nilpotent. Note that the reason for
excluding parameter values is usually obvious when considering the
matrix representing $\mathrm{ad}_{\e_6}$ whereas non-obvious
modifications are explained in footnotes. The names of the classes are
modified such that the remaining parameters are written as exponents
of the class symbol $\name$ and are denoted by $\a$, $\b$, $\c$ if
continuous and by $\ve$ if discrete.

The Lie brackets in both tables are written in the well-known dual
notation. Thereby, $\e^1,\dots, \e^6$ is a basis of $\g^*$ and the
images of $\e^i$ for $i=1,\ldots,6$ under the exterior derivative $d$
are given with rising $\mathrm i$ from left to right. We use the
abbreviation $\e^{ij}$ for $\e^i \wedge \e^j$. In the column labelled
$\mathfrak z$ the dimension of the centre of the corresponding Lie
algebra is given. The column labelled $\mathrm{h}^*(\g)$ contains the
vector $(\mathrm{h}^1(\g),\dots, \mathrm{h}^6(\g))$ of the dimensions
of the Lie algebra cohomology groups, where $\mathrm{h}^0(\g)$ is
omitted since it always equals one. Notice that unimodular Lie
algebras are characterised by the non-vanishing of the top-dimensional
cohomology group. In order to emphasise the unimodular entries in the
lists, the non-zero $\mathrm{h}^6(\g)$ are written bold and
underlined. The last column, labelled half-flat, is checked if and
only if the Lie algebra under consideration admits a half-flat
$\SU(3)$-structure. Note that all Lie algebras in Table
\ref{table_6dnonsolv} admit half-flat $\SU(3)$-structures.

Table \ref{examples_6dnonsolv} contains one example $(\omega,\rho)\in
\L^2 \g^*\times \L^3 \g^*$ of a normalised half-flat
$\SU(3)$-structure for each non-solvable indecomposable
six-dimensional Lie algebra. Similarly, Table \ref{examples6d_5dNR}
contains one example $(\omega,\rho)\in \L^2 \g^*\times \L^3 \g^*$ of a
normalised half-flat structure for each indecomposable six-dimensional
Lie algebra with five-dimensional nilradical which admits such a
structure. These structures are given in both cases in the
corresponding basis $(e^1,\ldots,e^6)$ of $\g^*$ of Table
\ref{table_6dnonsolv} or \ref{table_6d_5dNR}, respectively.  Moreover,
the Euclidean metric induced by these forms on $\g$ is added. The
label ONB indicates that the considered basis is
orthonormal. Similarly, OB indicates that the considered basis is
orthogonal. In this case, the norms of the non-unit basis vectors are
given explicitly.

Note that the collection of invariants in Table \ref{table_6d_5dNR} is
complemented by \cite{CS} where the invariants of the coadjoint
representation for the Lie algebras in question are determined.

\clearpage

\footnotesize
\renewcommand{\arraystretch}{1.5} \setlength{\tabcolsep}{0.1cm}
\setlength{\LTcapwidth}{16cm}
\begin{longtable}[ht]{lL{10cm}ccc}
  \caption{Non-solvable indecomposable six-dimensional Lie
    algebras} \\
 \hline $\g$ & Lie bracket & $\mathfrak{z}$ & $\mathrm{h}^*(\g)$ & half-flat
  \\ \hline \endhead
  \label{table_6dnonsolv}
  $L_{6,1}$ & $(\e^{23}, -\e^{13}, \e^{12}, \e^{26}-\e^{35},
  -\e^{16}+\e^{34}, \e^{15}-\e^{24})$ & 0 & (0,0,2,0,0,\fett 1) & \yes \\
  $L_{6,2}$ &  $(\e^{23}, 2\e^{12}, -2\e^{13}, \e^{14}+\e^{25}, -\e^{15}+\e^{34}, \e^{45})$ & 1 & (0,0,2,0,0,\fett 1) & \yes \\
  $L_{6,3}$ &  $(\e^{23}, 2\e^{12}, -2\e^{13}, \e^{14}+\e^{25}+\e^{46}, -\e^{15}+\e^{34}+\e^{56}, 0)$ & 0 & (1,0,1,1,0,0) & \yes \\
  $L_{6,4}$ &  $(\e^{23}, 2\e^{12}, -2\e^{13}, 2\e^{14}+2\e^{25}, \e^{26}+\e^{34}, -2\e^{16}+2\e^{35})$ & 0 & (0,0,2,0,0,\fett 1) & \yes \\
  $\mathfrak{so}(3,1)$ &  $(\e^{23}-\e^{56}, -\e^{13}+\e^{46}, \e^{12}-\e^{45}, \e^{26}, \e^{34}, \e^{15})$ & 0 & (0,0,2,0,0,\fett 1) & \yes \\ 
\end{longtable}

\renewcommand{\arraystretch}{1.5} \setlength{\tabcolsep}{0.3cm}
\begin{longtable}[ht]{lL{14cm}}
\caption{Half-flat $\SU(3)$-structures on non-solvable indecomposable six-dimensional Lie
    algebras} \\
\hline Lie algebra & Normalised half-flat $\SU(3)$-structure \\ \hline 
\endfirsthead
\caption{Half-flat $\SU(3)$-structures on non-solvable indecomposable six-dimensional Lie
    algebras -- continued}\\
\hline Lie algebra & Normalised half-flat $\SU(3)$-structure \\ \hline 
\endhead
\label{examples_6dnonsolv}
$L_{6,1}$, $\mathfrak{so}(3,1)$\quad & $\o=\e^{14}+\e^{25}-\e^{36}$, \, $\rho=-\e^{126}-\e^{135}+\e^{234}+\e^{456}$, ONB \\
$L_{6,2}$ & $\o=-\e^{14}+\e^{25}+\e^{36}$, \, $\rho=-2\e^{125}-2\e^{126}+3\e^{135}+2\e^{136}+\e^{156}+2\e^{234}+\e^{245}+\e^{345}-\e^{346}-\frac{19}{4}\e^{456}$, \\
& $g=2(\e^{1})^2+4(\e^{2})^2+6(\e^{3})^2+\frac{17}{2}(\e^{4})^2+\frac{53}{4}(\e^{5})^2+\frac{19}{2}(\e^{6})^2-8\e^{1}\ccdot\e^{4}-8\e^{2}\ccdot\e^{3}-4\e^{2}\ccdot\e^{6}-10\e^{3}\ccdot\e^{5}-4\e^{3}\ccdot\e^{6}+21\e^{5}\ccdot\e^{6}$ \\
$L_{6,3}$ &  $\o=\e^{15}+\e^{24}-\e^{26}+\e^{36}+\e^{56}$, \, $\rho=3\e^{124}+11\e^{126}-\e^{134}-2\e^{136}-2\e^{156}+\e^{235}-\e^{246}+\e^{346}+\e^{456}$, \\
& $g=5(\e^{1})^2+14(\e^{2})^2+(\e^{3})^2+(\e^{4})^2+(\e^{5})^2+18(\e^{6})^2-4\e^{1}\ccdot\e^{4}-18\e^{1}\ccdot\e^{6}-6\e^{2}\ccdot\e^{3}-4\e^{2}\ccdot\e^{5}+8\e^{4}\ccdot\e^{6}$ \\
$L_{6,4}$ &  $\o=-\frac{1}{2}\sqrt{3}\,( \e^{14}+\e^{15}-2\e^{24}-\e^{25}+\e^{36})$, \, $\rho=\frac{1}{2}\sqrt{3}\,(\e^{123}+\e^{126}+\e^{134}+\e^{235}-\e^{456})$, \\
& $g=(\e^{1})^2+2(\e^{2})^2+(\e^{3})^2+2(\e^{4})^2+(\e^{5})^2+(\e^{6})^2-2\e^{1}\ccdot\e^{2}+\e^{1}\ccdot\e^{4}+\e^{1}\ccdot\e^{5}-2\e^{2}\ccdot\e^{4}-\e^{2}\ccdot\e^{5}+\e^{3}\ccdot\e^{6}+2\e^{4}\ccdot\e^{5}$ \\
\end{longtable}

\renewcommand{\arraystretch}{1.5} \setlength{\tabcolsep}{0.1cm}
\begin{longtable}[ht]{lL{11.5cm}ccc}
  \caption{Indecomposable six-dimensional Lie algebras with five-dimensional non-Abelian nilradical} \\
  \hline $\g$ & Lie bracket & $\mathfrak z$ & $\mathrm{h}^*(\g)$ & half-flat
  \\ \hline 
  \endfirsthead
  \caption{Indecomposable six-dimensional Lie algebras with
      five-dimensional non-Abelian nilradical -- continued}\\
  \hline $\g$ & Lie bracket & $\mathfrak z$ & $\mathrm{h}^*(\g)$ & half-flat
  \\\hline 
  \endhead
  \label{table_6d_5dNR}

& \multicolumn{3}{c}{Nilradical $\h_3\op\bR^2$} & \\\hline

$\name^{\a,\b,\c}_{6,13}$ & $((\a+\b)\e^{16}+\e^{23}, \a\e^{26}, \b\e^{36}, \e^{46}, \c\e^{56}, 0)$  & && \\

\footnote{$\name^{\a,\b,\c}_{6,13} \cong \name^{\b,\a,\c}_{6,13} \cong \name^{\a/\c,\b/\c,1/\c}_{6,13}$, $\name^{0,0,\c}_{6,13}$ and $\name^{\a,\b,0}_{6,13}$ are decomposable.} 

& $0 < |\a| \leq |\b|$, $-1 < \c \leq 1$, \;$\a \neq -1$, \;$\b \notin \left\{-1,-\a,-2\a,-(2\a+1), -\frac{1}{2}(\a+1), -(\a+\frac{1}{2}) \right\}$, \;$\c \notin \left\{ 0,-(\a+1), -(\b+1), -(2\a+\b+1), -(2\a+2\b+1)\right\}$ &0& (1,0,0,0,0,0) & -- \\

& $\a=0$, $\b \notin \left\{-1,-\frac{1}{2},0\right\}$, $-1<\c \leq 1$, $\c\notin \left\{ 0, -\b,-2\b,-\b-1,-2\b-1\right\}$  &0& (2,1,0,0,0,0) & -- \\
& $\a=-1$, $\b\notin \left\{-1,0,\frac{1}{2},1,2\right\}$, $\c\notin \left\{-1,0,1,-\b,-2\b,-\b-1,-\b+1,-\b+2,-2\b+1,-2\b+2\right\}$ &0& (1,1,1,0,0,0) & -- \\
& $\c=-1$, $0 < |\a| \leq \b$, $\a\neq \pm 1$, $\b \notin\left\{1,-\a,-2\a,-2\a\pm 1,-\frac{1}{2}\a \pm \frac{1}{2},-\a\pm \frac{1}{2}\right\}$ &0& (1,1,1,0,0,0) & -- \\
& $\b=-2\a$, $\a \notin\left\{-1,0,\frac{1}{3},\frac{1}{2}\right\}$, $-1<\c\leq 1$, $\c \notin\left\{0,-\a,2\a,3\a,-1-\a,-1+2\a,-1+3\a\right\}$ &0& (1,1,1,0,0,0) & -- \\
& $\b=-(2\a+1)$, $\a\notin\left\{-1,-\frac{2}{3},-\frac{1}{2},-\frac{1}{3},0\right\}$, $\c\notin\left\{-1,0,1,-\a-1,2\a,2\a+2,3\a+1,3\a+2\right\}$  &0& (1,0,1,1,0,0) & -- \\
\end{longtable}

\setlength{\LTcapwidth}{16cm}
\renewcommand{\arraystretch}{1.55} \setlength{\tabcolsep}{0.15cm} 
\begin{longtable}[ht]{L{3.3cm}l}
  \caption{Indecomposable six-dimensional Lie
    algebras with five-dimensional nilradical admitting a half-flat
    $\SU(3)$-structure} \\
  \hline Lie algebra & Normalised half-flat $\SU(3)$-structure \\\hline
  \endfirsthead
  \caption{Indecomposable six-dimensional Lie
    algebras with five-dimensional nilradical admitting a half-flat
    $\SU(3)$-structure} \\
  \hline Lie algebra & Normalised half-flat $\SU(3)$-structure \\\hline
  \endhead
  \endfoot
  \endlastfoot
\label{examples6d_5dNR}
$\name_{6,13}^{\a,-(2\a+1),3\a+1}$, $\a\neq 0$ &
$\o=-\e^{13}+2\a\e^{26}+\e^{45}$,
$\rho=\e^{124}+2\a\e^{156}-\e^{235}-2\a\e^{346}$ , OB, $\left\|e_6\right\|=2 |\a|$ \\
$\name_{6,13}^{0,-1,1}$ &
$\o=\e^{15}-\e^{26}+\e^{34}$,
$\rho=-\e^{124}-\e^{136}-\e^{235}-\e^{456}$, ONB \\
$\name_{6,13}^{\a,-(2\a+1),3\a+2}$, $\a\neq-1$ &
$\o=-(2\a+2)\e^{16}-\e^{23}+\e^{45}$,
$\rho=\e^{124}+\e^{135}+(2\a+2)\e^{256}-(2\a+2)\e^{346}$, OB,
$||e_6||=2|\a+1|$ \\
$\name_{6,13}^{-1,1,-1}$ & $\o=-\e^{16}-\e^{25}+\e^{34}$,
$\rho=-\e^{124}-\e^{135}+\e^{236}-\e^{456}$, ONB \\
$\name_{6,14}^{\frac{2}{5},-\frac{3}{5}}$ &
$\o=-\e^{13}+\frac{4}{5}\e^{26}+\e^{45}$,
$\rho=\e^{125}-\frac{4}{5}\e^{146}+\e^{234}-\frac{4}{5}\e^{356}$, OB,
$||e_6||=\frac{4}{5}$ \\
$\name_{6,15}^{-1}$ & $\o=\e^{16}-\e^{24}+\e^{35}$,
$\rho=-\e^{125}-\e^{134}-\e^{236}-\e^{456}$, ONB \\
$\name_{6,18}^{-3,5}$ & $\o=4\e^{16}+\e^{23}+\e^{45}$,
$\rho=\e^{125}+\e^{134}+4\e^{246}-4\e^{356}$, OB,
$||e_6||=4$ \\
$\name_{6,18}^{2,-5}$ & $\o=-\e^{13}+4\e^{26}+\e^{45}$,
$\rho=-\e^{125}+4\e^{146}-\e^{234}+4\e^{356}$, OB,
$||e_6||=4$ \\
\end{longtable}

\end{document}